\documentclass[reqno]{amsart}
\numberwithin{equation}{section}
\usepackage{mathtools}
\usepackage{esint}
\usepackage{undertilde}
\newcommand\norm[1]{\left\lVert#1\right\rVert}
\usepackage{tikz}
\usetikzlibrary{arrows} 
\usetikzlibrary{calc,angles,quotes}
\newtheorem{theorem}{Theorem}[section]

\newtheorem{lemma}[theorem]{Lemma}

\usepackage{color,soul}
\usepackage{bm}

\title[Relative bending energy for weakly prestrained shells]{Relative bending energy for weakly prestrained shells}

\author{Silvia Jim\'enez Bola\~nos}
\address{Department of Mathematics, Colgate University,
13 Oak Drive, Hamilton, NY 13346 USA}
\email{sjimenez@colgate.edu}

\author{Anna Zemlyanova}
\address{Department of Mathematics, Kansas State University,
Manhattan, Kansas, 66506 USA }
\email{azem@ksu.edu }

\date{}


\begin{document}
\begin{abstract}
In this paper, we derive a dimensionally reduced model for a thin film prestrained with a given incompatible Riemannian metric: 
$$G^h(x',x_3)=I_3+2h^{\gamma}\,S(x')+2h^{\gamma/2}\,x_3B(x')+h.o.t, \,\,\,\gamma>2,$$ 
where $0<h\ll 1$ is the thickness of the film.  The problem is studied rigorously by using a variational approach and establishing the $\Gamma$-convergence of the non-Euclidean version of the nonlinear elasticity functional.  It is shown that the residual nonlinear elastic energy scales as $O(h^{\gamma+2})$ as $h\to 0$. 
\end{abstract}

\maketitle

\section{Introduction.}

The present study is concerned with the derivation of dimensionally reduced models for limiting behavior of thin films prestrained with a family of incompatible metrics $G^h$.  The motivation behind this study comes from applications for thin objects with internal prestrain such as growing tissues and various manufactured phenomena (for instance, polymer gels, atomically thin graphene layers, and plastically strained sheets).  Shape formation driven by internal prestrain is a very active area of research which has been tackled before by various  analytic and numerical arguments; see for instance \cite{Rodriguez1994,Klein2007,EFRATI2009762,Kimetal2012,Jones2015, Maor2014,KUPFERMAN2014,lewicka2019,lewicka2018}. 

The dimension reduction problems for thin plates and shells consist of minimizing a nonlinear elastic energy functional representing a mismatch between the deformation of the film and the target Riemannian metric $G$.  In \cite{Raoult1995, Raoult1996}, the authors considered the case $G=\mbox{Id}_3$ and derived nonlinear membrane models, for planar membranes and shells, from the variational formulation of the  three-dimensional nonlinear elasticity.  They showed that the deformations that minimize, or almost minimize, the total energy weakly converge in a Sobolev space towards deformations that minimize a nonlinear membrane energy, as the thickness of the body goes to zero.  The limiting nonlinear membrane energy was obtained by $\Gamma$-convergence of the sequence of three-dimensional energies.  The notion of $\Gamma$-convergence, introduced by Ennio De Giorgi in a series of papers published between 1975 and 1983 \cite{DeGiorgi84}, has become the standard notion of convergence for variational problems \cite{Braides02}.

In the fundamental papers \cite{FRM2002, Friesecke2006}, the authors derived the hierarchy of limiting theories of thin plates which arise as $\Gamma$- limits of three dimensional nonlinear theory in the classical elasticity.  In particular, they were able to show that von K\'arm\'an equations can be obtained rigorously from variational formulation of the nonlinear elasticity. The limiting theories differ in their scaling as powers of the thickness $h$ of the plate or shell depending on the scaling of the applied external forces.

The work in this area has been further extended to limiting theories for thin shells in \cite{Raoult1996, Friesekeetal2003, Lewickaetal2009, Lewickaetal2010, lewicka2011, Lewickaetal2011, LewickaPakzad2013, Lewickaetal2017}.  The first results for thin nonlinear elastic shells were obtained in \cite{Raoult1996} for the scaling $\beta=0$.  The case  $\beta=2$  was investigated in \cite{Friesekeetal2003}.  The nonlinear theory for shells with varying thickness with external loading was studied in \cite{Lewickaetal2009}.  In \cite{Lewickaetal2010}, the authors studied higher order $\beta\geq 4$ energy scaling for thin nonlinear elastic shells.  The paper \cite{Lewicka2011a} further extends the results of \cite{Lewickaetal2010}, by proving that the equilibria of nonlinear energy functional for a thin shell converge to the equilibria of the von K\'arm\'an functional.  The intermediate case of $2<\beta<4$ was studied in 
\cite{Lewickaetal2011}.  Finally, the results were summarized and the conjecture about the infinite hierarchy of limiting two-dimensional models was made in \cite{LewickaPakzad2013}.  The problem for thin shells has been revisited in \cite{Lewickaetal2017} to study the $\Gamma$-limit for thin shells with certain scaling of the applied forces.  In the papers mentioned above, the deformation of the plate or shell has been activated due to an application of external forces and not due to the prestrain.

In the context of the prestrain-driven response, the parallel theories are differentiated by the embeddability properties of the target metrics and, a-posteriori, by the emergence of isometry constraints on deformations with low regularity.  In turn, results on thin limit models have ramifications for the three dimensional original model with regard to energy scaling laws, understanding of the role of curvature in identifying the material's mechanical properties, and in the consequences of the symmetry and the symmetry breaking in the solutions to the resulting Euler-Lagrange equations.  The first work to rigorously study non-trivial configurations of thin prestrained flat films was produced in \cite{lewicka2011-c}. In this paper, the variational formulation of the problem was introduced by using a nonlinear elasticity functional and the necessary and sufficient conditions for existence of a $W^{2,2}$ isometric immersion were obtained.  The targeted metric in \cite{lewicka2011-c} does not depend on the thickness variable of the plate. Further results for quantization of the elastic energy for the case of thickness-independent metric have been obtained in \cite{Bhattacharya2016} and \cite{lewicka2017}.  A survey of available models for the thickness-independent prestrain is given in \cite{lewicka2018-r}. The question of immersability for a thickness-dependent Riemann metric satisfying certain general conditions has been decisively answered in \cite{lewicka2019}.  In \cite{lewicka2018}, the dimension reduction for oscillatory metrics was obtained.  The prescribed incompatibility metric in this paper exhibits a nonlinear dependence on the transversal variable. Asymptotic theories for shallow shells prestrained with a growth tensor linearly dependent on the transversal variable were studied in \cite{lewicka2014}. A survey of results available in prestrained elasticity can be found in \cite{LewickaPakzad2016}. 

In our paper, as in \cite{schmidt2007,lewicka2011,lewicka2015}, the prestrain metrics $G^h$ (see (\ref{Gh-N})) are perturbations of the flat $I_3$ metric.  In particular, in \cite{lewicka2015} the authors derived a new variational model consisting of minimizing a biharmonic energy of the out-of plane displacement $v\in W^{2,2}(\Omega,\mathbb{R})$, satisfying the Monge-Amp\`ere constraint $\det \nabla^2v=f,$ where $f=-{\rm curl}^T{\rm curl}\,S_{2\times2}$ is the linearized Gauss curvature of the Riemannian metrics in (\ref{Gh-N}).  The work in \cite{lewicka2015} was done in the parameter range $0<\gamma<2$, whereas the case $\gamma=2$ was treated in \cite{lewicka2011}, leading to the derivation of the F\"oppl-von K\'arm\'an equations accounting for the presence of the prestrain.  The main contribution of our paper is to develop the analysis for the parameter range $\gamma>2$.   We identified  the asymptotic behavior of the minimizers of $I_W^h(u^h)$ as $h\rightarrow0$ (see (\ref{I})), through deriving the $\Gamma$-limit of the rescaled energies $\displaystyle \frac{1}{h^{\gamma+2}}I_W^h(u^h)$.  These new outcomes are presented in Theorem~\ref{Th1} and Theorem~\ref{Th-recovseq}. A generalization of the current results to the growth tensors $A^h(x',x_3)=I_3+h^{\alpha}S(x')+h^{\gamma/2}x_3B(x')$, with arbitrary powers $\alpha$ and $\gamma$, is currently in preparation \cite{BLZ2019}. 

This paper is organized as follows.  In Section~\ref{sec:Form}, we give the problem formulation.  In Section~\ref{sec:gamma}, we present the main results of this paper, which are stated in Theorem~\ref{Th1} and Theorem~\ref{Th-recovseq}.  Finally, in Section~\ref{sec:proof1} and Section~\ref{sec:proof2}, we present the proofs of Theorem~\ref{Th1} and Theorem~\ref{Th-recovseq}, respectively.

\section{Problem formulation: non-Euclidean elasticity model.}
\label{sec:Form}
Let the smooth invertible growth tensors $A^h=[A_{ij}^{h}]:\overline{\Omega^h}\rightarrow\mathbb{R}^{3\times 3}$, $\det A^h>0$, be defined by: 
\begin{equation}
	\label{Ah-N}
	A^h(x',x_3)=I_3+h^{\gamma}S(x')+h^{\gamma/2}x_3B(x'),
\end{equation}
where the scaling exponent $\gamma>2$ and $I_n$ is the $n\times n$ identity matrix.  Consider a family of three-dimensional thin plates: $$\Omega^h=\omega\times\big(-h/2,h/2\big)\subset\mathbb{R}^3,$$ where $\omega$ is an open bounded set of $\mathbb{R}^2$ and $0<h\ll1$, viewed as the reference configurations of thin elastic films.  A typical point in $\Omega^h$ is denoted by $x=(x_1,x_2,x_3)=(x',x_3)$, where $x'\in\omega$ and $|x_3|<h/2$.  The ``stretching" and ``bending" tensors $S,B:\overline{\omega}\rightarrow\mathbb{R}^{3\times3}$ are two given smooth matrix fields.

	For a deformation $u^h:\Omega^h\rightarrow\mathbb{R}^3$, its elastic energy $I_W^h(u^h)$, defined by:
\begin{align}
	\label{I}
	I_W^h(u^h)&=\frac{1}{h}\int_{\Omega^h}W(F)dx\notag\\
	&=\frac{1}{h}\int_{\Omega^h}W(\nabla u^h(A^h)^{-1})dx \hspace{4mm}\forall u^h\in W^{1,2}(\Omega^h,\mathbb{R}^3),
\end{align} 
is given in terms of the elastic tensor $F=\nabla u^h(A^h)^{-1}$ (see \cite{Rodriguez1994}) accounting for the reorganization of $\Omega^h$ in response to $A^h$.

	The elastic energy density $W:\mathbb{R}^{3\times 3}\rightarrow\mathbb{R}_+$ is assumed to satisfy the standard conditions of normalization, frame indifference with respect to the special orthogonal group $SO(3)$ of proper rotations in $\mathbb{R}^3$, and second order nondegeneracy given by: 
\begin{align}
	\label{W1}
		\exists\, c>0  \hspace{2mm} \forall\, F\in\mathbb{R}^{3\times 3}\hspace{3mm} \forall \,R\in SO(3)\hspace{3mm} & W(R)=0, \hspace{3mm} W(RF)=W(F),\notag\\
	& W(F)\geq c\hspace{1mm} {\text{dist}}^2(F,SO(3)).
\end{align}
	We also assume that there exists a monotone nonnegative function $\nu:[0,+\infty]\rightarrow[0,+\infty]$ which converges to zero at $0$, and a quadratic form $\mathcal{Q}_3$ on $\mathbb{R}^{3\times 3}$, with:
\begin{equation}
	\label{SOND}
	\forall F\in\mathbb{R}^{3\times 3} \hspace{5mm} \big|W(I_3+F)-\frac{1}{2}\mathcal{Q}_3(F)\big|\leq\nu(|F|)|F|^2.
\end{equation}
	If $W$ is $\mathcal{C}^2$ regular in a neighborhood of $SO(3)$, then condition (\ref{SOND}) is satisfied and, in that case, we have $\mathcal{Q}_3=D^2W(I_3)$.  Note that (\ref{SOND}) implies that $\mathcal{Q}_3$ is nonnegative, is positive definite on symmetric matrices, and $\mathcal{Q}_3(F)=\mathcal{Q}_3(\text{sym}\hspace{1mm} F)$ for all $F\in\mathbb{R}^{3\times 3}$ (see \cite{lewicka2015} for a proof).

	Recall that in (\ref{I}), $I_W^h(u^h)=0$ is equivalent, via  (\ref{W1}) and the polar decomposition theorem, to:
\begin{equation}
	\label{I0}
	(\nabla u^h)^T\nabla u^h=(A^h)^T(A^h) \hspace{4mm} \text{and} \hspace{4mm} \text{det}\nabla u^h>0 \hspace{3mm} \text{in } \Omega^h;
\end{equation}
which is equivalent to: $I_W^h(u^h)=0$ if and only if $u^h$ is an isometric immersion of the Riemannian metric $G^h=(A^h)^T(A^h)$.  Therefore, the quantity:
\begin{equation*}
	\label{RE}
	e_h=\inf\left\{I_W^h(u^h);\hspace{1mm}u^h\in W^{1,2}(\Omega^h,\mathbb{R}^3)\right\}
\end{equation*}
measures the residual energy at free equilibria of the configuration $\Omega^h$ that has been prestrained by $G^h$.  This is consistent with Theorem~2.2 in \cite{lewicka2011-c}, which observes that $e_h>0$ whenever $G^h$ has no smooth isometric immersion in $\mathbb{R}^3$, i.e. when there is no $u^h$ with (\ref{I0}) or, equivalently, when the Riemann curvature tensor of the metric $G^h$ does not vanish identically on $\Omega^h$.

	Observe now that $A^h$ in (\ref{Ah-N}) yields:
\begin{align}
	\label{Gh-N}
		G^h(x',x_3)&=(A^h)^T(A^h)\notag\\
	&=I_3+2h^{\gamma}\text{sym}\hspace{1mm} S(x')+2h^{\gamma/2}x_3\text{sym}\hspace{1mm}B(x')\notag\\
	&\quad+2h^{3\gamma/2}x_3\text{sym}\,S(x')^TB(x')\notag\\
	&\quad+h^{2\gamma}S(x')^TS(x')+h^{\gamma}x_3^2B(x')^TB(x').
\end{align}

\subsection{Notation.}
For a matrix $F$, $F_{n\times m}$ denotes its $n\times m$ principal minor.  If $m=n$, the symmetric part of a square matrix $F$ is denoted by ${\rm sym}F=(F+F^T)/2$.  The superscript $^T$ refers to the transpose of a matrix or an operator.  

	Also, for any $F\in \mathbb{R}^{2\times2}$, we denote by $F^*\in \mathbb{R}^{3\times3}$ the matrix for which $F^*_{2\times2}=F$ and $F^*_{i3}=F^*_{3i}=0$, for $i=1,2,3$.  By $\nabla_{tan}$ we denote taking derivatives $\partial_1$ and $\partial_2$ in the in-plate directions $e_1=(1,0,0)^T$ and $e_2=(0,1,0)^T$.  The derivative $\partial_3$ is taken in the out-of-plate direction $e_3=(0,0,1)^T$.

	Finally, we use the Landau symbols $\mathcal{O}(h^\alpha)$ and $o(h^\alpha)$ to denote quantities which are of the order of, or vanish faster than $h^\alpha$, as $h\rightarrow0$.  By $C$ we denote any universal constant, depending on $\omega$ and $W$, but independent of other involved quantities, so that $C=\mathcal{O}(1)$, and it can change values from line to line.

\subsection{Formal discussion about scaling.}
	Consider the deformations $u^h:\Omega^h\rightarrow\mathbb{R}^3$ of the shell $\Omega^h$ given by: $$u^{h}(x',x_3)=(x',0)^T+h^{\gamma/2} V(x')+x_3N_h(x')+h.o.t.,$$ where $N_h(x')$ is the unit normal to the midplate and $V:\omega\rightarrow\mathbb{R}^3$.  Notice that we have:
\begin{equation*}
\label{gradu}
\nabla u^h(x',x_3)=(I_3)^*+h^{\gamma/2}\nabla V+ N_h\otimes e_3+x_3\nabla N_h+h.o.t.,
\end{equation*}
and, since $F=\nabla u^h (A^h)^{-1}$, we have that:
\begin{align}
	\label{RFTF-2}
	\sqrt{F^TF}&=I_3-h^\gamma{\rm sym}\, S-x_3h^{\gamma/2} {\rm sym}\,B+h^{\gamma/2}\,\text{sym}\nabla V\notag\\
	&\quad+\frac{h^\gamma}{2}{\nabla V}^T\nabla V+x_3\text{sym}\left(\left[I_3+h^{\gamma/2}{\nabla V}^T\right]\nabla N_h\right)+\text{h.o.t.}
\end{align}
Now, by (\ref{RFTF-2}) and using the Taylor expansion for $W$, we can go back to (\ref{I}) to obtain:
\begin{equation}
	\label{I2}
	I_W^h(u^h)=\frac{1}{h}\int_{\Omega^h}W(\sqrt{F^TF})dx'dx_3 =\frac{1}{h}\int_{\Omega^h}\left(\frac{1}{2}\mathcal{Q}_3(\star)+h.o.t.\right)dx'dx_3,
\end{equation}
where $\mathcal{Q}_3(\star)=D^2W(Id)(\star,\star)$, and $\star$ is given by:
\begin{align*}
\star&=-h^\gamma{\rm sym} S-x_3h^{\gamma/2} {\rm sym}B+h^{\gamma/2}\,\text{sym}\nabla V\\
	&\quad+\frac{h^\gamma}{2}\nabla^TV\nabla V+x_3\text{sym}\left(\left[I_3+h^{\gamma/2}\nabla V^T)\right]\nabla N_h\right)+\text{h.o.t}.
\end{align*}

Thus, we can rewrite (\ref{I2}) as:
\begin{align*}
	\label{I3}
	I_W^h(u^h)&=\int_{\omega}\frac{1}{2}\mathcal{Q}_3(-h^\gamma {\rm sym}S+h^{\gamma/2}\,\text{sym}\nabla V+\frac{h^\gamma}{2}{\nabla V}^T\nabla V)dx'\notag\\
	&\quad+\frac{h^2}{24}\int_{\omega}\mathcal{Q}_3(-h^{\gamma/2} {\rm sym}B-h^{\gamma/2}\nabla^2V_3+h^\gamma \mathcal{D})dx'+h.o.t.\notag\\
	&=\frac{h^{2\gamma}}{8}\int_{\omega}\mathcal{Q}_3(-2{\rm sym}S+{\nabla V}^T\nabla V)dx'\notag\\
	&\quad+\frac{h^{\gamma+2}}{24}\int_{\omega}\mathcal{Q}_3(-{\rm sym}B-\nabla^2V_3+h^{\frac{\gamma}{2}} \mathcal{D})dx'+h.o.t.,
\end{align*} 
where the matrix $D$ is given by:
\begin{equation*}
	\mathcal{D}=\begin{pmatrix}V_{3,1}V_{1,11}+V_{3,2}V_{2,11}&V_{3,1}V_{1,12}+V_{3,2}V_{2,12}\\V_{3,1}V_{1,21}+V_{3,2}V_{2,21}&V_{3,1}V_{1,22}+V_{3,2}V_{2,22}\end{pmatrix}.
\end{equation*}
Since $\gamma>2$, we observe that $I_W^h(u^h)\approx C\,h^{\gamma+2}$ and hence, we expect the $\Gamma$-limit of $\displaystyle \frac{1}{h^{\gamma+2}}I_W^h(u^h)$ to be only the first order change in the linear bending energy: $$\frac{1}{24}\int_{\omega}\mathcal{Q}_3(B+\nabla^2V_3)\,dx'.$$

\section{The variational limit in the case $\gamma>2$.}
\label{sec:gamma}
As explained before, the main goal of this paper is the identification of the asymptotic behavior of the minimizers of $I_W^h(u^h)$ as $h\rightarrow0$, through deriving the $\Gamma$-limit of the rescaled energies $\displaystyle \frac{1}{h^{\gamma+2}}I_W^h(u^h)$.

In \cite{lewicka2015} the authors derived a new variational model consisting of minimizing a biharmonic energy $\displaystyle\int_\omega |\nabla^2v|^2~dx'$ of the out-of plane displacement $v\in W^{2,2}(\omega,\mathbb{R})$, satisfying the Monge-Amp\`ere constraint $\det \nabla^2v=f,$ where $f=-{\rm curl}^T{\rm curl}\,S_{2\times2}$ is the linearized Gauss curvature of the Riemannian metrics in (\ref{Gh-N}).  This work was done in the parameter range $0<\gamma<2$, whereas the case $\gamma=2$ was treated in \cite{lewicka2011}, leading to the derivation of the F\"oppl-von K\'arm\'an equations accounting for presence of the prestrain.  In what follows, we carry out the analysis for the parameter range $\gamma>2$.  

We now state the main results of this paper:
\begin{theorem}
	\label{Th1}
	Let $A^h$ be given as in (\ref{Ah-N}), with an arbitrary exponent $\gamma>2$.  Assume that a sequence of deformations $u^h\in W^{1,2}(\Omega^h,\mathbb{R}^3)$ satisfies:
\begin{equation}
\label{Scaling}
	I_W^h(u^h)\leq Ch^{\gamma+2},
\end{equation}
where $W$ fulfills (\ref{W1}) and (\ref{SOND}).  Then, there exist rotations $\bar{R}^h\in SO(3)$ and translations $c^h\in\mathbb{R}^3$ such that, for the normalized deformations:
\begin{equation}
	\label{deform}
	y^h\in W^{1,2}(\Omega^1,\mathbb{R}^3), \hspace{3mm} y^h(x',x_3)=(\bar{R}^h)^Tu^h(x',hx_3)-c^h,
\end{equation}
the following hold (up to a subsequence that we do not relabel):
\begin{itemize}
	\item[(i)] $y^h(x',x_3)\rightarrow x'$ in $W^{1,2}(\Omega^1,\mathbb{R}^3)$
	\item[(ii)] The scaled displacements 
\begin{equation}
	\label{scaled-disp}
	V^h(x')=\frac{1}{h^{\gamma/2}}\fint_{-1/2}^{1/2}y^h(x',t)-x'dt
\end{equation} 
converge to a vector field $V$ of the form $V=(0,0,V_3)^T$.  This convergence is strong in $W^{1,2}(\omega,\mathbb{R}^3)$.  The only non-zero out-of-plane scalar component $V_3$ of $V$ belongs to $W^{2,2}(\omega,\mathbb{R})$.  
	\item[(iii)] Moreover: 
\begin{equation}
	\label{lowerboundRescaledEnergy}
	\liminf_{h\rightarrow0}\frac{1}{h^{\gamma+2}}I_W^h(u^h)\geq\mathcal{I}_\gamma(V_3),
\end{equation}
where $\mathcal{I}_\gamma:W^{2,2}(\omega)\rightarrow\bar{\mathbb{R}}_+$ is given by:
\begin{equation}
	\label{If}
	\mathcal{I}_\gamma=\frac{1}{24}\int_{\omega}\mathcal{Q}_2(\nabla^2 V_3+\left({\rm sym}B(x')\right)_{2\times2})\,dx',
\end{equation}
and the quadratic non-degenerate form $\mathcal{Q}_2$, acting on matrices $F\in\mathbb{R}^{2\times2}$ is given by:
\begin{equation}
	\label{Q2}
	\mathcal{Q}_2(F)=\min\left\{\mathcal{Q}_3(\tilde{F}):\tilde{F}\in\mathbb{R}^{3\times3},\,\tilde{F}_{2\times2}=F \right\}.
\end{equation}
\end{itemize}
\end{theorem}

Now, for the optimality of the energy bound in (\ref{lowerboundRescaledEnergy}) and of the scaling (\ref{Scaling}):
\begin{theorem}
	\label{Th-recovseq}  Assume (\ref{Ah-N}), (\ref{W1}), and (\ref{SOND}).  Moreover, assume that $\omega$ is simply connected and that $\gamma>2$.  Then, for every $V_3\in W^{2,2}(\omega,\mathbb{R})$, there exists a sequence of deformations $u^h\in W^{1,2}(\Omega^h,\mathbb{R}^3)$ such that the following hold:
\begin{itemize}
	\item[(i)] The sequence $y^h(x',x_3)=u^h(x',hx_3)$ converge in $W^{1,2}(\Omega^1,\mathbb{R}^3)$ to $x'$.
	\item[(ii)] $V^h(x')=\displaystyle \frac{1}{h^{\gamma/2}}\fint_{-h/2}^{h/2}\left(u^h(x',t)-x'\right)\,dt$ converge in $W^{1,2}(\Omega,\mathbb{R}^3)$ to $(0,0,V_3)^T$.
	\item[(iii)] Recalling (\ref{If}), one has:
\begin{equation*}
	\lim_{h\rightarrow0}\frac{1}{h^{\gamma+2}}I_W^h(u^h)=\mathcal{I}_\gamma(V_3).
\end{equation*}
\end{itemize}
\end{theorem}

As a result of Theorem~{\ref{Th1}} and Theorem~{\ref{Th-recovseq}} we have:
\begin{theorem}
	\label{Th-Gamma-minim} 
	Assume (\ref{Ah-N}), (\ref{W1}), and (\ref{SOND}).  Moreover, assume that $\omega$ is simply connected and that $\gamma>2$.  Then there exist a uniform constant $C\geq0$ such that: $$e_h={\rm inf} I_W^h\leq Ch^{\gamma+2}.$$  Under this condition, for any minimizing sequence $u^h\in W^{1,2}(\Omega^h,\mathbb{R}^3)$ for $I_W^h$, i.e. when: 
\begin{equation}
	\label{limdiffIW-inf}
	\lim_{h\rightarrow0}\frac{1}{h^{\gamma+2}}\left(I_W^h(u^h)-{\rm inf}I_W^h\right)=0,
\end{equation}
the convergences (i), (ii) of Theorem~\ref{Th1} hold up to a subsequence, and the limit $V_3$ is a minimizer of the functional $I_\gamma$ defined as in (\ref{If}).

	Moreover, for any (global) minimizer $V_3$ of $I_\gamma$, there exists a minimizing sequence $u^h$, satisfying (\ref{limdiffIW-inf}) together with (i), (ii), and (iii) of Theorem~\ref{Th-recovseq}.
\end{theorem}

\section{Proof of Theorem~{\ref{Th1}}.}
\label{sec:proof1}
\begin{proof}
This proof will be separated in different sections.  In Section~\ref{subsec:bRh}, we obtain the rotations $\bar{R}^h$ required by Theorem~{\ref{Th1}}.  In Sections~\ref{subsec:iTh1}, \ref{subsec:iiTh1}, and \ref{subsec:iiiTh1} we present the proof of statements (i), (ii), and (iii) respectively.
\subsection{Construction of the rotations $\bar{R}^h$.} 
\label{subsec:bRh}
First, we quote the following approximation result (Theorem~\ref{FRM-thm}), which can be directly obtained from the geometric rigidity estimate found in Theorem~1.6 in \cite{FRM2002}, in view of the following bounds:
\begin{align*}
&\left.Var(A^h)=\norm{\nabla_{t}(A^h\right|_{x_3=0})}_{L^\infty(\omega)}+\norm{\partial_3A^h}_{L^\infty(\Omega^h)}\leq Ch^{\frac{\gamma}{2}},\\
&\norm{A^h}_{L^\infty(\Omega^h)}+\norm{(A^h)^{-1}}_{L^\infty(\Omega^h)}\leq C.\notag
\end{align*}
\begin{theorem}
\label{FRM-thm}
	Let $u^h\in W^{1,2}(\Omega^h,\mathbb{R}^3)$ satisfy $\displaystyle\lim_{h\rightarrow0}\frac{1}{h^2}I_W^h(u^h)=0$, which is, in particular, implied by (\ref{Scaling}).  Then, there exist matrix fields $R^h\in W^{1,2}(\omega,\mathbb{R}^{3\times3})$, such that $R^h(x')\in SO(3)$ for a.e. $x'\in\omega$, and:
\begin{equation}
	\label{rigid}
\frac{1}{h}\int_{\Omega^h}\left|\nabla u^h(x)-R^h(x')A^h(x)\right|^2dx\leq Ch^{2+\gamma}, \hspace{4mm}\int_{\omega}\left|\nabla R^{h}\right|^2\,dx'\leq Ch^\gamma.
\end{equation}
\end{theorem}
  In order to achieve the proof of compactness in Theorem~{\ref{Th1}}, we outline similar arguments as in \cite{lewicka2015}, emphasizing the differences and new outcomes of this paper.

	Assume (\ref{Scaling}) and let $R^h\in W^{1,2}(\omega,\mathbb{R}^{3\times3})$ be the matrix fields given by Theorem~\ref{FRM-thm}. Define the averaged rotations: $\tilde{R}^h=\mathbb{P}_{SO(3)}\displaystyle\fint_{\omega}R^h$.  These projections of $\displaystyle\fint_{\omega}R^h$ onto $SO(3)$ are well defined for small $h$ since, by (\ref{rigid}) and Poincar\'e's inequality, we have:
\begin{equation*}
{\rm dist}^2\big(\fint_{\omega}R^h,SO(3)\big)\leq C\fint_{\omega}\left|\nabla R^h\right|^2\,dx'\leq Ch^\gamma,
\end{equation*}
which together with Poincar\'e's inequality, deliver:
\begin{equation}
	\label{1projwelldef}
	\int_{\omega}\left|R^h-\tilde{R}^h\right|^2\leq C\big(\int_{\omega}\big|R^h-\fint_{\omega}R^h\big|^2+{\rm dist}^2\big(\fint_{\omega}R^h,SO(3)\big)\big)\leq Ch^\gamma.
\end{equation}
	Now, let:
\begin{equation}
	\label{projhat}
	\hat{R}^h=\mathbb{P}_{SO(3)}\fint_{\Omega^h}(\tilde{R}^h)^T\nabla u^h,
\end{equation}
which is well defined for small $h$ because:
\begin{align}
	\label{2projwelldef}
	{\rm dist}^2&\big(\fint_{\Omega^h}{(\tilde{R}^h)}^T\nabla u^h,SO(3)\big)\notag\\
	&\leq C\big(\fint_{\Omega^h}\left|\nabla u^h-R^hA^h\right|^2+\fint_{\Omega^h}\left|A^h-I_3\right|^2+\fint_{\Omega^h}|R^h-\tilde{R}^h|^2\big)\notag\\	&\leq Ch^\gamma,
\end{align}
where we used (\ref{rigid}), (\ref{Ah-N}), and (\ref{1projwelldef}).  Consequently, using the fact that $I_3=\mathbb{P}_{SO(3)}I_3$, we also obtain: 
\begin{align}
\label{hatR-id}
	|\hat{R}^h-I_3|^2&\leq C\big({\rm dist}^2\big(\fint_{\Omega^h}(\tilde{R}^h)^T\nabla u^h,SO(3)\big)+\big|\fint_{\Omega^h}(\tilde{R}^h)^T\nabla u^h-I_3\big|^2\big) \notag\\
&\leq Ch^\gamma. 
\end{align}
We may now define:
\begin{equation}
	\label{projbar}
	\bar{R}^h=\tilde{R}^h\hat{R}^h.
\end{equation}

By (\ref{1projwelldef}), (\ref{hatR-id}), and (\ref{rigid}), it follows that:
\begin{equation}
	\label{3projwelldef}
	\int_{\omega}\left|R^h-\bar{R}^h\right|^2\leq C\Big(\int_{\omega}|R^h-\tilde{R}^h|^2+\int_{\omega}|\tilde{R}^h(I_3-\hat{R}^h)|^2\Big)\leq Ch^\gamma.
\end{equation}

Now, consider:
\begin{equation}
	\label{limitId}
	\norm{(\bar{R}^h)^TR^h-I_3}^2_{W^{1,2}(\omega)}=\int_{\omega}\left|R^h-\bar{R}^h\right|^2+\int_{\omega}\left|\nabla R^h\right|^2\leq Ch^\gamma,
\end{equation}
and observe that from (\ref{limitId}) we conclude:
\begin{equation}	
	\label{limitRbarTR=I_3}
\lim_{h\rightarrow0}(\bar{R}^h)^TR^h=I_3 \hspace{3mm}\text{ in $W^{1,2}(\omega,\mathbb{R}^{3\times3})$}.
\end{equation}

\begin{lemma}
There exist vectors $c^h\in\mathbb{R}^3$ such that, for the re-scaled averaged displacement $V^h$ defined as in (\ref{scaled-disp}): $$V^h(x')=\displaystyle\frac{1}{h^{\gamma/2}}\fint_{-1/2}^{1/2}(\bar{R}^h)^Tu^h(x',ht)-c^h-x'dt,$$ it holds:
	\begin{equation}
	\label{rescaled-V}
		\int_{\omega}V^h\,dx'=0, \hspace{4mm}\text{skew}\int_{\omega}\nabla V^h\,dx'=0.
	\end{equation}
\end{lemma}	
\begin{proof}
To ensure the first statement in (\ref{rescaled-V}), we can take: $$c^h=\frac{1}{|\omega|}\int_{\omega}\fint_{-1/2}^{1/2}\left[(\bar{R}^h)^Tu^h(x',ht)-x'\right]dtdx'.$$
For the second statement in (\ref{rescaled-V}), we observe that: 
\begin{equation}
\label{nablaVh}
\nabla V^h=\frac{1}{h^{\gamma/2}}\fint_{-1/2}^{1/2}\left[(\bar{R}^h)^T\nabla_{tan}u^h(x',ht)-(I_{3})^*\right]dt.
\end{equation}
In view of (\ref{projbar}), we obtain that: $$\displaystyle(\bar{R}^h)^T\fint_{\Omega_h}\nabla u^h=(\hat{R}^h)^T\fint_{\Omega^h}(\tilde{R}^h)^T\nabla u^h;$$ and, because of (\ref{projhat}) and (\ref{2projwelldef}), $\displaystyle (\bar{R}^h)^T\fint_{\Omega_h}\nabla u^h$ is symmetric.  Hence: $$\displaystyle\text{skew}\fint_{\omega}\nabla V^h=\displaystyle{h^{-\gamma/2}}\text{skew}\fint_{\Omega^h}(\bar{R}^h)^T\nabla u^h=0.$$  In particular, we see as well that (\ref{projbar}) is equivalent to:
\begin{equation*}
	\label{projbar2}
	\bar{R}^h=\mathbb{P}_{SO(3)}\fint_{\Omega^h}\nabla u^h.
\end{equation*}
The above result is true since for a matrix $F$ sufficiently close to $SO(3)$, its projection $R_0=\mathbb{P}_{SO(3)}F$ coincides with the unique rotation appearing in the polar decomposition of $F$, that is, $F=R_0U$ with $\text{skew}\,U=0$. 
\end{proof}
\subsection{Proof of (i) in Theorem~\ref{Th1}.}
\label{subsec:iTh1}  We use (\ref{deform}), (\ref{projbar}), (\ref{2projwelldef}), and (\ref{hatR-id}), to obtain:
\begin{align}
	\label{proof-i-1}
	&\norm{\left(\nabla y^h-I_3\right)_{3\times2}}^2_{L^2(\Omega^1)}\notag\\
	&\quad\leq \frac{1}{h}\int_{\Omega^h}\left|(\bar{R}^h)^T\nabla u^h-I_3\right|^2\notag\\
	&\quad\leq C\Big(\frac{1}{h}\int_{\Omega^h}\big|(\tilde{R}^h)^T\nabla u^h-I_3\big|^2+\frac{1}{h}\int_{\Omega^h}\big|\hat{R}^h-I_3\big|^2\Big)\notag\\
	&\quad\leq Ch^\gamma,
\end{align}
and also, by (\ref{rigid}), we get:
\begin{align}
		\label{proof-i-2}
	\norm{\partial_3 y^h}^2_{L^2(\Omega^1)}&=h\int_{\Omega^h}\left|\partial_3 u^h\right|^2\notag\\
	&\leq Ch\Big(\int_{\Omega^h}\left|\nabla u^h-R^hA^h\right|^2+\int_{\Omega^h}\left|A^h\right|^2\Big)\notag\\
	&\leq Ch^2.
\end{align}
By the first statement in (\ref{rescaled-V}), we have:
\begin{equation*}
	\label{proof-i-3}
	\int_{\Omega^1}y^h(x)-x'\,dx=h^{\gamma/2}\int_{\omega}V^h\,dx'=0;
\end{equation*}
hence, using Poincar\'e's inequality, (\ref{proof-i-1}), and (\ref{proof-i-2}), we have:
\begin{align}
	\label{proof-i-4}
	\norm{y^h(x)-x'}^2_{L^2(\Omega^1)}&\leq C\int_{\Omega^1}\left|\nabla y^h-I_3\right|^2\notag\\
	&\leq C\left(\int_{\Omega^1}\left|\nabla y^h-I_3\right|_{3\times2}^2+\int_{\Omega^1}\left|\partial_3 y^h\right|^2\right)\notag\\
	&\leq Ch^2,
\end{align}
and therefore, we obtain the convergence of $y^h$ to $x'$ in $W^{1,2}(\Omega^1)$.

\subsection{Proof of (ii) in Theorem~\ref{Th1}.}
\label{subsec:iiTh1}
\begin{lemma}
$V^h$ converges (up to a subsequence) weakly in $W^{1,2}(\omega,\mathbb{R}^3)$.
\end{lemma} 
\begin{proof}Observe that, using the first statement in (\ref{rescaled-V}) and Poincar\'e's inequality, we have:
\begin{equation*}
	\label{proof-ii-1}
	\norm{V^h}^2_{L^2(\omega)}\leq C\norm{\nabla V^h}^2_{L^2(\omega)}.
\end{equation*}
Also, by Jensen's inequality and (\ref{proof-i-1}), we obtain:
\begin{align*}
	\label{proof-ii-2}
	\norm{\nabla V^h}^2_{L^2(\omega)}&=\Big|\Big|\frac{1}{h^{\gamma/2}}\fint_{-1/2}^{1/2}\left[(\bar{R}^h)^T\nabla_{tan}u^h(x',ht)-(I_{3})^*\right]dt\Big|\Big|^2_{L^2(\omega)}\notag\\
	&\leq \frac{C}{h^{\gamma}}\int_{\Omega^1}\left|(\bar{R}^h)^T\nabla_{tan}u^h(x',ht)-(I_{3})^*\right|^2\\
	&\leq C.
\end{align*}
Therefore:
\begin{equation*}
	\label{proof-ii-3}
	\norm{V^h}^2_{W^{1,2}(\omega)}=\norm{V^h}^2_{L^2(\omega)}+\norm{\nabla V^h}^2_{L^2(\omega)}\leq C.
\end{equation*}
Hence, $V^h$ is a bounded sequence in the norm of $W^{1,2}(\omega,\mathbb{R}^3)$, which implies the $V^h$ is weakly convergent (up to a subsequence, still called $V^h$) to $V$ in $W^{1,2}(\omega,\mathbb{R}^3)$.  
\end{proof}
	Consider the matrix fields $D^h\in W^{1,2}(\omega,\mathbb{R}^{3\times3})$:
\begin{align}
	\label{Dh}
	D^h(x')&=\frac{1}{h^{\gamma/2}}\fint_{-h/2}^{h/2}(\bar{R}^h)^TR^h(x')A^h(x',t)-I_3\, dt\notag\\
	&= h^{\gamma/2}(\bar{R}^h)^TR^h(x')\,S(x')+\frac{1}{h^{\gamma/2}}\left((\bar{R}^h)^TR^h(x')-I_3\right).
\end{align}
\begin{lemma}
\label{lemma:Dhconv}
The sequence $D^h$ converges (up to a subsequence) weakly in $W^{1,2}(\omega,\mathbb{R}^{3\times3})$.
\end{lemma} 
\begin{proof}
	Since S is smooth and by (\ref{3projwelldef}), we have that:
\begin{align*}
&\norm{D^h}_{L^2(\omega)}^2\\
&\quad=\int_{\omega}\big|h^{\gamma/2}(\bar{R}^h)^TR^h(x')\,S(x')+\frac{1}{h^{\gamma/2}}\left((\bar{R}^h)^TR^h(x')-I_3\right)\big|^2\,dx'\notag\\
	&\quad\leq C\left(h^{\gamma}\int_{\omega}\left|S(x')\right|^2\,dx'+\frac{1}{h^{\gamma}}\int_{\omega}\left|\left(R^h(x')-\bar{R}^h\right)\right|^2\,dx'\right)\\
	&\quad\leq C.
\end{align*}
Also, by (\ref{rigid}), we have that:
\begin{align*}
&\norm{\nabla D^h}_{L^2(\omega)}^2\\
&=\int_{\omega}\big|h^{\gamma/2}(\bar{R}^h)^T\nabla\left(R^h(x')\,S(x')\right)+\frac{1}{h^{\gamma/2}}(\bar{R}^h)^T\nabla R^h(x')\big|^2\,dx'\notag\\
	&\leq C\big(h^{\gamma}\int_{\omega}\left|\nabla R^h(x')\right|^2\left| S(x')\right|^2\,dx'\\
	&\quad+h^{\gamma}\int_{\omega}\left|\nabla S(x')\right|^2\,dx'+\frac{1}{h^{\gamma}}\int_{\Omega}\left|\nabla R^h\right|^2\big)\notag\\
	&\leq C.
\end{align*}
Therefore $\norm{D^h}^2_{W^{1,2}(\omega)}\leq C$ and so, up to a subsequence still called $D^h$, we obtain:
\begin{equation}
	\label{limitDh}
	\lim_{h\rightarrow0}D^h=D \hspace{4mm}\text{weakly in $W^{1,2}(\omega,\mathbb{R}^{3\times3})$}
\end{equation}
\end{proof}
The following limit:
\begin{equation}
	\label{limitDhq}
	\lim_{h\rightarrow0}\frac{1}{h^{\gamma/2}}\left(\left(\bar{R}^h\right)^TR^h-I_3\right)=D \hspace{4mm}\text{in $L^{q}(\omega,\mathbb{R}^{3\times3})$,}
\end{equation}
for all $q\geq1$, is obtained by noticing that:
\begin{align}
	\label{limitDhq-1}
	&\Big|\Big|\frac{1}{h^{\gamma/2}}\big((\bar{R}^h)^TR^h-I_3\big)-D\Big|\Big|_{L^q(\omega)}\notag\\
	&\quad\leq\Big|\Big|\frac{1}{h^{\gamma/2}}\big((\bar{R}^h)^TR^h-I_3\big)-D^h\Big|\Big|_{L^q(\omega)}+\norm{D^h-D}_{L^q(\omega)}\notag\\
	&\quad\leq h^{\gamma/2}\norm{(\bar{R}^h)^TR^h(x')\,S(x')}_{L^q(\omega)}+\norm{D^h-D}_{L^q(\omega)}.
\end{align}
Since $SO(3)$ is a bounded set, we know that $\norm{(\bar{R}^h)^TR^h}_{L^\infty(\omega)}\leq C$ for some $C>0$; and since $S$ is smooth, we have that $\norm{S}_{L^q(\omega)}$ is also bounded.  Then, the first term on the right hand side of (\ref{limitDhq-1}) is bounded by $Ch^{\gamma/2}$.  For the second term, by Lemma~\ref{lemma:Dhconv}, the Rellich-Kondrachov theorem, and since $\omega$ is bounded, we have that $D^h$ converges strongly in $L^q(\Omega)$ to $D$, for $q\geq1$ .  Therefore we have proven (\ref{limitDhq}).

\begin{lemma}
The limiting matrix field $D$ has skew-symmetric values:
\begin{equation}
	\label{symD=0}
	{\rm sym}\,D=\lim_{h\rightarrow0}{\rm sym}\,D^h=0.
\end{equation}
\end{lemma}
\begin{proof}
For all $R\in SO(3)$, we have:
\begin{equation}
	\label{identR}
	(R-I_3)^T(R-I_3)=2\,I_3-2\,{\rm sym}\,R=-2\,{\rm sym}(R-I_3).
\end{equation}
Now, using (\ref{identR}), (\ref{3projwelldef}), and (\ref{rigid}), we obtain:
\begin{align*}
&\frac{1}{h^{\gamma/2}}\norm{{\rm sym}\big((\bar{R}^h)^TR^h-I_3\big)}_{L^2(\omega)}\\
	&\quad=\frac{1}{2h^{\gamma/2}}\norm{\big((\bar{R}^h)^TR^h-Id\big)^T\big((\bar{R}^h)^TR^h-I_3\big)}_{L^2(\omega)}\notag\\
	&\quad\leq C\frac{1}{h^{\gamma/2}}\norm{R^h-\bar{R}^h}^2_{W^{1,2}(\omega)}\\
	&\quad\leq Ch^{\gamma/2}.
\end{align*}	
Then, it follows that: $$\frac{1}{h^{\gamma/2}}\norm{{\rm sym}\big((\bar{R}^h)^TR^h-I_3\big)}_{L^2(\omega)}\rightarrow0$$ as $h\rightarrow0$; and, as a result, $D$ has skew-symmetric values.  
\end{proof}

By (\ref{identR}), we observe that:
\begin{align}
	\label{identDhsym}
	\frac{1}{h^{\gamma/2}}{\rm sym}D^h&={\rm sym}\left[(\bar{R}^h)^TR^h(x')\,S(x')+\frac{1}{h^{\gamma}}\left((\bar{R}^h)^TR^h(x')-I_3\right)\right]\notag\\
	&={\rm sym}\left[(\bar{R}^h)^TR^h(x')\,S(x')\right]\notag\\
	&\quad-\frac{1}{2h^{\gamma}}\left((\bar{R}^h)^TR^h-I_3\right)^T\left((\bar{R}^h)^TR^h-I_3\right).
\end{align}

First, since $S$ is smooth, we have:
\begin{align}
	\label{lqSlim}
	\norm{(\bar{R}^h)^TR^h(x')\,S(x')-S(x')}_{L^q(\omega)}&=\norm{\left[(\bar{R}^h)^TR^h(x')-I_3\right]S(x')}_{L^q(\omega)}\notag\\
	&\leq C\norm{(\bar{R}^h)^TR^h(x')-I_3}_{L^q(\omega)}.
\end{align}
Using (\ref{limitId}) and the Rellich-Kondrachov theorem, it follows that the limit, as $h\rightarrow0$, of the right hand side of (\ref{lqSlim}) is zero, for $q\geq1$.

Also, observe that:
\begin{align}
	\label{limsecpartbel}
	&\Big|\Big|\frac{1}{h^{\gamma}}\left((\bar{R}^h)^TR^h-I_3\right)^T\left((\bar{R}^h)^TR^h-I_3\right)+D^2\Big|\Big|_{L^q(\omega)}\notag\\
	&\leq\Big|\Big|\frac{1}{h^{\gamma/2}}\left((\bar{R}^h)^TR^h-I_3\right)^T\Big[\frac{1}{h^{\gamma/2}}\left((\bar{R}^h)^TR^h-I_3\right)-D\Big]\Big|\Big|_{L^q(\omega)}\notag\\
	&\quad+\Big|\Big|\Big[\frac{1}{h^{\gamma/2}}\left((\bar{R}^h)^TR^h-I_3\right)-D\Big]^TD\Big|\Big|_{L^q(\omega)}
\end{align}
approaches $0$, as $h\rightarrow0$, by (\ref{limitDhq}).  

Therefore, by (\ref{identDhsym}), (\ref{lqSlim}) and (\ref{limsecpartbel}), we obtain:
\begin{align*}
	\lim_{h\rightarrow0}\frac{1}{h^{\gamma/2}}{\rm sym}D^h&=\lim_{h\rightarrow0}{\rm sym}\big[(\bar{R}^h)^TR^h(x')\,S(x')\big]\\
	&\quad-\frac{1}{2h^{\gamma}}\left((\bar{R}^h)^TR^h-I_3\right)^T\left((\bar{R}^h)^TR^h-I_3\right)\notag\\
	&={\rm sym}\,S+\frac{1}{2}D^2 \hspace{4mm}\text{ in $L^q(\omega,\mathbb{R}^{3\times3})$, \,\,$\forall q\geq1$.}
\end{align*}

As for the convergence of $V^h$, we have by (\ref{nablaVh}) and (\ref{Dh}) that:
\begin{align}
	\label{nablaVh-Dh}
&\big[\nabla V^h(x')-D^h(x')\big]_{3\times2}\notag\\
	&=\frac{1}{h^{\gamma/2}}(\bar{R}^h)^T\fint_{-h/2}^{h/2}\left[\nabla_{tan}u^h(x',t)-R^h(x')A^h(x',t)\right]_{3\times2}dt,
\end{align}
which, together with (\ref{rigid}), imply that:
\begin{align}
	\label{nablaVh-Dh-2}
	&\norm{\big[\nabla V^h(x')-D^h(x')\big]_{3\times2}}^2_{L^2(\omega)}\notag\\
	&\leq\frac{C}{h^{\gamma+1}}\int_{\Omega^h}\left|\nabla u^h(x',t)-R^h(x')A^h(x',t)\right|^2dx\notag\\
	&\leq Ch^2,
\end{align}
and therefore, by (\ref{limitDh}), the sequence $\nabla V^h$ converges in $L^2(\omega,\mathbb{R}^{3\times2})$ to $D$.  
Observe that:
\begin{equation}
	\label{nablaV}
	\nabla_{tan} V=D_{3\times2}, 
\end{equation}
and $D\in W^{1,2}(\omega,\mathbb{R}^{3\times3})$ by (\ref{limitDh}).  Then, we obtain $V\in W^{2,2}(\omega,\mathbb{R}^3).$

Finally, we use (\ref{symD=0}) to conclude that ${\rm sym}(\nabla_{tan} V)_{2\times2}=0$, and so, by Korn's inequality:
\begin{equation*}
	\int_{\omega}\left|(\nabla_{tan} V)_{2\times2}\right|^2dx'\leq C\int_{\omega}\left|{\rm sym}(\nabla_{tan} V)_{2\times2}\right|^2dx'=0,
\end{equation*}
which implies $V_{tan}$ must be constant and, hence, equal to $0$ in view of (\ref{rescaled-V}).  Then $V=\left(0,0,V_3\right)$.  This ends the proof of (ii) in Theorem~\ref{Th1}.

\subsection{Proof of (iii) in Theorem~\ref{Th1}.}
\label{subsec:iiiTh1} To prove the lower bound (\ref{lowerboundRescaledEnergy}), we define the rescaled strains $P^h\in L^2(\Omega^1,\mathbb{R}^{3\times3})$ by:
\begin{equation*}
	\label{Ph}
	P^h(x',x_3)=\frac{1}{h^{\gamma/2+1}}\left((R^h(x'))^T\nabla u^h(x',hx_3)\left(A^h(x',hx_3)\right)^{-1}-I_3\right).
\end{equation*}
\begin{lemma}
\label{lemma:Phconv}
The rescaled strains $P^h$ converge (up to a subsequence) weakly in $L^{2}(\Omega^1,\mathbb{R}^{3\times3})$.
\end{lemma} 
\begin{proof}
	Observe that, since $A^h$ is smooth, $R^h\in SO(3)$, and using (\ref{rigid}), we have:
\begin{equation}
	\label{L2boundPh}
	\norm{P^h}^2_{L^2(\Omega^1)}\leq\frac{C}{h^{\gamma+3}}\int_{\Omega^h}\left|\nabla u^h(x',t)-A^h(x',t)\right|^2dx'dt\leq C.
\end{equation}
Then, up to a subsequence, we get:
\begin{equation}
	\label{limPh}
	\lim_{h\rightarrow0}P^h=P\hspace{5mm} \text{weakly in $L^2(\Omega^1,\mathbb{R}^{3\times3})$.}
\end{equation}
\end{proof}
	We follow similar steps as in \cite{lewicka2011} to obtain a useful property of the limiting strain $P$.  First, by (\ref{deform}), we obtain:
\begin{align}
\label{convP}
\frac{\left(\partial_3y^h-he_3\right)}{h^{\gamma/2+1}}&=\frac{1}{h^{\gamma/2}}(\bar{R}^h)^T\left[\nabla u^h(x',hx_3)-R^h(x')A^h(x',hx_3)\right]e_3\notag\\
&\quad+\frac{1}{h^{\gamma/2}}(\bar{R}^h)^TR^h(x')\left[A^h(x',hx_3)-I_3\right]e_3\notag\\
&\quad+\frac{1}{h^{\gamma/2}}\left[(\bar{R}^h)^TR^h(x')-I_3\right]e_3.
\end{align}

Let's study all three terms on the right-hand side of (\ref{convP}).  For the first term, using (\ref{rigid}), we get:
\begin{align*}
\frac{1}{h^\gamma}&\int_{\Omega^1}\left|(\bar{R}^h)^T\left[\nabla u^h(x',hx_3)-R^h(x')A^h(x',hx_3)\right]e_3\right|^2dx'dx_3\\
&\leq\frac{1}{h^{\gamma+1}}\int_{\Omega^h}\left|\nabla u^h(x',t)-R^h(x')A^h(x',t)\right|^2dx'dx_3\notag\\
&\leq Ch^2.
\end{align*}
For the second term, we have:
\begin{align*}
\frac{1}{h^\gamma}&\int_{\Omega^1}\left|(\bar{R}^h)^TR^h(x')\left[A^h(x',hx_3)-I_3\right]e_3\right|^2dx'dx_3\\
&\leq\int_{\Omega^1}\left|h^{\gamma/2} S(x')+hx_3B(x')\right|^2dx'dx_3\\
	&\leq Ch^2.
\end{align*}
And finally, for the third term, using (\ref{limitDhq}), we have that it converges to $De_3$ in $L^2(\omega)$.  Therefore:
\begin{equation}
	\label{limconvP}
	\lim_{h\rightarrow0}\frac{1}{h^{\gamma/2+1}}\left(\partial_3y^h-he_3\right)=De_3 \hspace{3mm}\text{ in $L^2(\Omega^1, \mathbb{R}^{3})$}.
\end{equation}

	For each small $s>0$, we define the sequence of functions $f^{s,h}\in W^{1,2}(\Omega^1,\mathbb{R}^3)$:
\begin{equation}
	\label{fsh}
	f^{s,h}(x)=\frac{1}{h^{\gamma/2+1}}\frac{1}{s}(y^h(x+se_3)-y^h(x)-hse_3).
\end{equation}
Clearly, (\ref{fsh}) is equivalent to $$f^{s,h}(x)=\displaystyle\frac{1}{h^{\gamma/2+1}}\fint_0^s(\partial_3y^h(x+te_3)-he_3)dt,$$ and, by (\ref{limconvP}), we have that:
\begin{equation}
\label{limfsh-2}
\lim_{h\rightarrow0}f^{s,h}=De_3 \hspace{3mm}\text{ in $L^2(\Omega^1,\mathbb{R}^3)$}.
\end{equation}
Also, observe that $\partial_3f^{s,h}(x)=\displaystyle\frac{1}{s}\frac{1}{h^{\gamma/2+1}}\left(\partial_3y^h(x+se_3)-\partial_3y^h(x)\right)$.  Then:
\begin{align*}
	\norm{\partial_3f^{s,h}(x)}_{L^2(\Omega^1)}&\leq\frac{1}{|s|}\norm{\frac{1}{h^{\gamma/2+1}}\left(\partial_3y^h(x+se_3)-he_3\right)-De_3}_{L^2(\Omega^1)}\\
&\quad+\frac{1}{|s|}\norm{\frac{1}{h^{\gamma/2+1}}\left(\partial_3y^h(x)-he_3\right)-De_3}_{L^2(\Omega^1)}
\end{align*}
goes to $0$ as $h\rightarrow0$, by (\ref{limconvP}).    In other words:
\begin{equation}
\label{limpartialfsh}
\lim_{h\rightarrow0}\partial_3f^{s,h}=0 \hspace{3mm}\text{ in $L^2(\Omega^1,\mathbb{R}^3)$}.
\end{equation}
Further, for any $\alpha=1,2$, we have:
\begin{align*}
\partial_\alpha f^{s,h}(x)&=\frac{1}{s}(\bar{R}^h)^TR^h(x')\big[P^h(x',x_3+s)A^h(x',hx_3+hs)\notag\\
&\quad-P^h(x',x_3)A^h(x',hx_3)\\
&\quad+\frac{1}{h^{\gamma/2+1}}\left(A^h(x',hx_3+hs)-A^h(x',hx_3)\right)\big]e_\alpha;
\end{align*}
which, in view of (\ref{limPh}), (\ref{limitRbarTR=I_3}), and since $A^h\rightarrow I_3$ strongly in $L^2(\omega,\mathbb{R}^{3\times3})$, yields the weak convergence in $L^2(\Omega^1,\mathbb{R}^{3\times2})$ of $\partial_\alpha f^{s,h}(x)$:
\begin{align}
\label{limpartialalphafsh}
\lim_{h\rightarrow0}\partial_\alpha f^{s,h}(x)&=\frac{1}{s}\left(P(x',x_3+s)-P(x',x_3)\right)e_\alpha\notag\\
&\quad+\frac{1}{s}\lim_{h\rightarrow0}\frac{1}{h^{\gamma/2+1}}\left(I_3+h^\gamma S(x')+h^{\gamma/2+1}(x_3+s)B(x')\right.\notag\\
	&\quad\left.-(I_3+h^\gamma S(x')+h^{\gamma/2+1}x_3B(x'))\right)e_\alpha\notag\\
&=\frac{1}{s}\left(P(x',x_3+s)-P(x',x_3)\right)e_\alpha+B(x')e_\alpha.
\end{align}
Consequently, by (\ref{limfsh-2}), (\ref{limpartialfsh}), and (\ref{limpartialalphafsh}), we see that $f^{s,h}$ converges weakly in $W^{1,2}(\omega,\mathbb{R}^3)$ to $De_3$.  Hence, the left-hand side in (\ref{limpartialalphafsh}) equals $\partial_\alpha(De_3)$:
\begin{equation*}
\label{matchDpartialalphafsh}
	\partial_\alpha(De_3)=\frac{1}{s}\left(P(x',x_3+s)-P(x',x_3)\right)e_\alpha+B(x')e_\alpha.
\end{equation*}
Also, we notice that:
\begin{equation*}
\label{partialPalpha}
	(\partial_3 P)e_\alpha=\lim_{s\rightarrow0}\Big(\frac{P(x',x_3+s)-P(x',x_3)}{s}\Big)e_\alpha=\partial_\alpha(De_3)-B(x')e_\alpha;
\end{equation*}
and consequently:
\begin{equation}
\label{P3-2}
	P(x)_{3\times2}=\left(\nabla(D(x')e_3)-B(x')\right)_{3\times2}x_3+P_0(x')_{3\times2},
\end{equation}
for some $P_0\in L^2(\omega,\mathbb{R}^{3\times3})$.

We can now finish the proof of Theorem~\ref{Th1}.   Observe that, from (\ref{SOND}), and using the Taylor expansion of the function $W(F)$ close to $F=I_3$, we obtain:
\begin{align}
\label{limrescaledW}
\frac{1}{h^{\gamma+2}}&W\left(\nabla u^h(x)(A^h(x))^{-1}\right)\notag\\
&=\frac{1}{h^{\gamma+2}}W\left((R^h(x))^T\nabla u^h(x)(A^h(x))^{-1}\right)\notag\\
&=\frac{1}{h^{\gamma+2}}W\left(Id+h^{\gamma/2+1}P^h(x)\right)\notag\\
&=\frac{1}{2}\mathcal{Q}_3(P^h(x))+\nu\left(h^{\gamma/2+1}|P^h|\right)\mathcal{O}(|P^h(x)|^2).
\end{align}
Consider the sets $\mathcal{U}_h=\left\{x\in\Omega^1\,:\,h\left|P^h(x',x_3)\right|\leq1\right\}$.  Note that, using H\"older's inequality and (\ref{rigid}), we have:
\begin{align}
	\label{hPh1}
	&\int_{\Omega^1}|hP^h|dx\notag\\
	&\quad=\frac{1}{h^{\gamma/2+1}}\int_{\Omega^h}\left|(R^h(x'))^T\nabla u^h(x',t)\left(A^h(x',t)\right)^{-1}-Id\right|dx\notag\\
	&\quad\leq\frac{C}{h^{\gamma/2+1}}\Big(\int_{\Omega^h}\left|\nabla u^h(x',hx_3)-R^h(x')A^h(x',hx_3)\right|^2dx\Big)^{1/2}\notag\\	
	&\quad\leq Ch^{1/2}.
\end{align}
Since, by (\ref{hPh1}), $hP^h$ converges to $0$ in $L^1(\Omega^1)$ as $h\rightarrow0$, then there exists a subsequence (which we call again $hP^h$) such that $hP^h$ converges to $0$ pointwise a.e.

Now, for that subsequence observe that, by (\ref{hPh1}), we obtain:
\begin{align*}
	&\int_{\Omega^1}|\chi_{\mathcal{U}_h}-1|^2dx=\int_{\Omega^1\setminus \mathcal{U}_h}1dx\\
	&\quad\leq \int_{\Omega^1\setminus \mathcal{U}_h}|hP^h|dx\leq \int_{\Omega^1}|hP^h|dx\leq Ch^{1/2}.
\end{align*}
And therefore, $\chi_{\mathcal{U}_h}$ converges to $1$ in $L^2(\Omega^1)$ as $h\rightarrow0$.

	Since $\displaystyle \lim_{t\rightarrow0}\nu(t)=0$, using (\ref{limrescaledW}), and the properties of $\mathcal{Q}_3$, we get:
\begin{align}
	\label{lb-1}
	&\liminf_{h\rightarrow0}\frac{1}{h^{\gamma+2}}I_W^h(u^h)\notag\\
&\quad\geq\liminf_{h\rightarrow0}\frac{1}{h^{\gamma+2}}\int_{\Omega^1}\chi_{\mathcal{U}_h}W\Big(\nabla u^h(x',hx_3)\left(A^h(x',hx_3)\right)^{-1}\Big)dx\notag\\
	&\quad=\liminf_{h\rightarrow0}\Big(\frac{1}{2}\int_{\Omega^1}\mathcal{Q}_3(\chi_{\mathcal{U}_h}P^h(x))dx+o(1)\int_{\Omega^1}|P^h(x)|^2dx\Big)\notag\\
	&\quad\geq\frac{1}{2}\int_{\Omega^1}\mathcal{Q}_3({\rm sym}P(x))dx,
\end{align}
where the convergence to $0$ of the term $\displaystyle o(1)\int_{\Omega^1}|P^h(x)|^2dx$ as $h\rightarrow0$ follows from (\ref{rigid});  and since, by (\ref{limPh}), $\chi_{\mathcal{U}_h}P^h(x)$ converges weakly to $P$ in $L^2(\Omega^1,\mathbb{R}^{3\times3})$.  

	Further, by (\ref{Q2}) and (\ref{P3-2}):
\begin{align}
	\label{lb-2}
	\frac{1}{2}&\int_{\Omega^1}\mathcal{Q}_3({\rm sym}P(x))dx\notag\\
	&\geq
\frac{1}{2}\int_{\Omega^1}\mathcal{Q}_2({\rm sym}P_{2\times2}(x))dx\notag\\
	&=\frac{1}{2}\int_{\Omega^1}\mathcal{Q}_2(x_3\,{\rm sym}\left(\nabla De_3-B(x')\right)_{2\times2}+{\rm sym}P_0(x')_{2\times2})dx\notag\\
	&\geq\frac{1}{24}\int_{\Omega}\mathcal{Q}_2({\rm sym}\left(\nabla De_3\right)_{2\times2}-\left({\rm sym}B(x')\right)_{2\times2})\,dx'.
\end{align}
Now, in view of Theorem~\ref{Th1}~(ii), (\ref{symD=0}), and (\ref{nablaV}) we have: $$(De_3)^T=\langle D_{1,3}, D_{2,3}, 0\rangle^T=\langle -D_{3,1}, -D_{3,2}, 0\rangle^T=-\langle \partial_1V_3, \partial_2V_3, 0\rangle^T,$$ in other words: $$(\nabla De_3)_{2\times2}=-\nabla^2V_3,$$ which yields the claim in Theorem~\ref{Th1}~(iii), by (\ref{lb-1}) and (\ref{lb-2}).
\end{proof}
\section{Proof of Theorem~\ref{Th-recovseq}.}
\label{sec:proof2}
\begin{proof}Recalling (\ref{Q2}) the definition of $\mathcal{Q}_2$, let $c(F)\in\mathbb{R}^3$ be the unique vector so that: $$\mathcal{Q}_2(F)=\mathcal{Q}_3(F^*+{\rm sym}(c\otimes e_3)).$$ 
The mapping $c:\mathbb{R}^{2\times2}_{\rm sym}\rightarrow\mathbb{R}^3$ is well-defined and linear, by the properties of $\mathcal{Q}_3$.   Also, for all $F\in \mathbb{R}^{3\times3}$, we denote by $l(F)$ the unique vector in $\mathbb{R}^3$, linearly depending on $F$, such that: $${\rm sym}(F-(F_{2\times2})^*)={\rm sym}(l(F)\otimes e_3).$$

	Let the out-of-plane displacement $V_3$ be as in Theorem~\ref{Th1}.  By the Rellich-Kondrachov embedding theorem, since $V_3\in W^{2,2}(\omega,\mathbb{R})$, we have $V_3\in W^{1,q}(\omega,\mathbb{R})$ for all $1\leq q<\infty$.  We first prove the result under the additional assumption of $V_3$ being smooth up to the boundary in Section~\ref{sec:smooth}.  In Section~\ref{sec:W22}, we prove the result for $V_3\in W^{2,2}(\omega,\mathbb{R})$.

\subsection{Case: $V_3\in C^{\infty}(\bar{\omega},\mathbb{R})$.}
\label{sec:smooth}  Define the recovery sequence:
\begin{align}
\label{recseq}
	&u^h(x',x_3)\notag\\
	&=\begin{bmatrix}x' \\ 0 \end{bmatrix}+\begin{bmatrix}0\\ h^{\gamma/2}V_3(x') \end{bmatrix}+x_3\begin{bmatrix}-h^{\gamma/2} (\nabla_{tan} V_3(x'))^T\\ 1 \end{bmatrix}+\frac{1}{2}h^{\gamma/2}x_3^2d^1(x'),
\end{align}
for $(x',x_3)\in\Omega^h$, where the smooth warping field $d^1:\bar{\Omega}\rightarrow\mathbb{R}^3$ is given by:
\begin{equation}
	\label{d1}
d^1=l(B)+c\left(-\nabla^2 V_3-({\rm sym} B)_{2\times2}\right).
\end{equation}
Calculating the deformation gradient, we obtain:
\begin{equation*}
	\label{nablauh-recseq}
	\nabla u^h=I_3+h^{\gamma/2}\mathcal{V}-h^{\gamma/2}x_3\left(\nabla_{tan}^2V_3\right)^*+h^{\gamma/2}\begin{bmatrix}\frac{1}{2}x_3^2\nabla_{tan} d^{1} & x_3d^1\end{bmatrix},
\end{equation*}
where the skew-symmetric matrix field $\mathcal{V}$ is given by
\begin{equation*}
\label{ssV}
	\mathcal{V}=\begin{bmatrix}0&-(\nabla_{tan} V_3)^T\\\nabla_{tan} V_3 & 0\end{bmatrix}.
\end{equation*}
The convergence statements in (i) and (ii) of Theorem~\ref{Th-recovseq} are verified by a straightforward calculation:

	For (i):
\begin{equation*}
	\norm{y^h(x',x_3)-x'}_{W^{1,2}(\Omega^1,\mathbb{R}^3)}\leq Ch.
\end{equation*}

	For (ii):
\begin{equation*}
	\norm{V^h(x')-\begin{bmatrix}0&0&V_3\end{bmatrix}^T}_{W^{1,2}(\Omega,\mathbb{R}^3)}\leq Ch^2.
\end{equation*}

 To prove (iii), we need to estimate the energy of the sequence $u^h$.   In order to do this, we shall use an auxiliary $SO(3)$-valued matrix $R^h=e^{h^{\gamma/2}\mathcal{V}}$.  Clearly, $R^h=I_3+h^{\gamma/2}\mathcal{V}+\frac{h^\gamma}{2}\mathcal{V}^2+h.o.t$ and $(R^h)^T=I_3-h^{\gamma/2}\mathcal{V}+\frac{h^\gamma}{2}\mathcal{V}^2+h.o.t$.  Also, recall that $(A^h)^{-1}=Id-h^{\gamma}S-h^{\gamma/2}x_3B+h.o.t$.  Hence, we obtain: 
\begin{align*}
&(R^h)^T(\nabla u^h)(A^h)^{-1}\\
&=I_3+h^\gamma\Big(-\frac{1}{2}\mathcal{V}^2-S\Big)+h^{\gamma/2}x_3\left(-\left(\nabla^2V_3\right)^*-B+d^1\otimes e_3\right)+h.o.t.\notag
\end{align*}

	Using the definition of the quadratic form $\mathcal{Q}_3(F)$, Taylor expanding the energy density $W$ around the identity, and taking into account the uniform boundedness of all the involved functions and their derivatives, we get in (\ref{I}):
\begin{align}
	\label{energyrs}
	I_W^h(u^h)&=\frac{1}{h}\int_{\Omega^h}W((R^h)^T(\nabla u^h)(A^h)^{-1})dx\notag\\
	&=\frac{h^{2\gamma}}{2}\int_{\Omega}\mathcal{Q}_3\Big({\rm sym}\big(-\frac{1}{2}\mathcal{V}^2-S\big)\Big)dx'\notag\\
	&\quad+\frac{h^{\gamma+2}}{24}\int_{\Omega}\mathcal{Q}_3\left({\rm sym}\left(-\left(\nabla^2V_3\right)^*-B+d^1\otimes e_3\right)\right)dx'+h.o.t.
\end{align}
In view of (\ref{energyrs}), it follows that:
\begin{align*}
	\label{energyrs-2}
	&\frac{1}{h^{\gamma+2}}I_W^h(u^h)\\
	&\quad=\frac{1}{24}\int_{\Omega}\mathcal{Q}_3\left({\rm sym}\left(-\left(\nabla^2V_3\right)^*-B+d^1\otimes e_3\right)\right)dx'+O(h^{\gamma-2}).
\end{align*}

	Observe that, from (\ref{d1}),we get that:
\begin{align*}
	&{\rm sym}\left(-\left(\nabla^2V_3\right)^*-B+d^1\otimes e_3\right)\\
	&=\left(-\nabla^2V_3-({\rm sym}B)_{2\times2}\right)^*+{\rm sym}\left(\left(d^1-l(B)\right)\otimes e_3\right)\\
&=\left(-\nabla^2V_3-({\rm sym}B)_{2\times2}\right)^*+{\rm sym}\left(\left(c\left(-\nabla^2 V_3-({\rm sym} B)_{2\times2}\right)\right)\otimes e_3\right),
\end{align*}
which implies:
\begin{equation}
	\label{energyrs-3}
	\frac{1}{h^{\gamma+2}}I_W^h(u^h)=\mathcal{I}_f(V_3)+O(h^{\gamma-2}),
\end{equation}
which, in turn, proves (iii) for smooth displacement $V_3$.

\subsection{Case: $V_3\in W^{2,2}(\omega,\mathbb{R})$.}
\label{sec:W22}
We now consider a sequence $V_3^n\in C^\infty(\overline{\omega},\mathbb{R})$ such that:
\begin{equation}
	\label{smoothapprox}
	\norm{V_3^n-V_3}_{W^{2,2}(\omega,\mathbb{R})}\rightarrow0, \hspace{3mm}\text{ as $n\rightarrow\infty$}.
\end{equation}
Define the sequence $u^{h,n}$ as in (\ref{recseq}) using $V_3^n$ instead of $V_3$:
\begin{align*}
	&u^{h,n}(x',x_3)\\
	&=\begin{bmatrix}x' \\ 0 \end{bmatrix}+\begin{bmatrix}0\\ h^{\gamma/2}V_3^n(x') \end{bmatrix}+x_3\begin{bmatrix}-h^{\gamma/2} (\nabla V_3^n(x'))^T\\ 1 \end{bmatrix}+\frac{1}{2}h^{\gamma/2}x_3^2d^{1,n}(x'),
\end{align*}
where the smooth warping field $d^{1,n}:\bar{\omega}\rightarrow\mathbb{R}^3$ is given by:
\begin{equation*}
	\label{d1n}
d^{1,n}=l(B)+c\left(-\nabla^2 V_3^n-({\rm sym} B)_{2\times2}\right).
\end{equation*}
Calculating the deformation gradient, we first obtain:
\begin{equation*}
	\label{nablauh-recseq-n}
	\nabla u^{h,n}={\rm I_3}+h^{\gamma/2}\mathcal{V}^n-h^{\gamma/2}x_3\left(\nabla^2V_3^n\right)^*+h^{\gamma/2}\begin{bmatrix}\frac{1}{2}x_3^2\nabla d^{1,n} & x_3d^{1,n}\end{bmatrix},
\end{equation*} 
where the skew-symmetric matrix field $\mathcal{V}^h$ is given by:
\begin{equation*}
\label{ssVn}
	\mathcal{V}^n=\begin{bmatrix}0&-(\nabla V_3^n)^T\\\nabla V_3^n & 0\end{bmatrix}.
\end{equation*}

	Let the sequence $\left\{n(h)\right\}$ be such that $n(h)\rightarrow\infty$ when $h\rightarrow0$, and define:
\begin{equation*}
\label{recseq-n(h)}
	u_h(x',x_3):=u^{h,n(h)}(x',x_3).
\end{equation*}
Observe that:
\begin{align}
	\label{trineq}
&\left|\frac{1}{h^{\gamma+2}}I_W^h(u_h)-\mathcal{I}_f(V_3)\right|\notag\\
&\quad\leq\left|\frac{1}{h^{\gamma+2}}I_W^h(u_h)-\mathcal{I}_f(V_3^{n(h)})\right|+\left|\mathcal{I}_f(V_3^{n(h)})-\mathcal{I}_f(V_3)\right|.
\end{align}
We will study the two terms on the right hand side of (\ref{trineq}).  As above, we use an auxiliary $SO(3)$-valued matrix $R^{h,n(h)}=e^{h^{\gamma/2}\mathcal{V}^{n(h)}}$.  Hence, we obtain: 
\begin{align*}
	&\norm{I_3-(R^{h,n(h)})^T(\nabla u_h)(A^h)^{-1}}_{L^{\infty}(\Omega^h,\mathbb{R}^{3\times3})}\\
	&\leq C\,h^{\gamma/2+1} \Big(1+\big|\big|\nabla {V_3}^{n(h)}\big|\big|_{L^{\infty}(\omega)}+\big|\big|\big(\nabla^2{V_3}^{n(h)}\big)^*\big|\big|_{L^{\infty}(\omega)}\\
	&\quad+\big|\big|\nabla^3{V_3}^{n(h)}\big|\big|_{L^{\infty}(\omega)}\Big).
\end{align*}
Let $\epsilon>0$, we can always find $h_1>0$ such that for all $h\leq h_1$: 
\begin{align}
&h^{\gamma/2+1} \Big(1+\big|\big|\nabla {V_3}^{n(h)}\big|\big|_{L^{\infty}(\omega)}+\big|\big|\big(\nabla^2{V_3}^{n(h)}\big)^*\big|\big|_{L^{\infty}(\omega)}\notag\\
	&\quad+\big|\big|\nabla^3{V_3}^{n(h)}\big|\big|_{L^{\infty}(\omega)}\Big)<\frac{\epsilon}{2C},
\end{align}
by reparameterizing the sequence $V_3^{n(h)}$, in order to slow down the rate of convergence to $V_3$.  

Then, in a similar way to (\ref{energyrs-3}), we obtain:
\begin{align*}
	&\frac{1}{h^{\gamma+2}}I_W^h(u_h)\\
	&=\mathcal{I}_f(V_3^n(n))+O\Big(\min(h^{\gamma-2},h^2)F\big(\nabla {V_3}^{n(h)},(\nabla^2{V_3}^{n(h)})^*,\nabla^3{V_3}^{n(h)}\big)\Big),\notag
\end{align*}
where the quantity $F\left(\nabla {V_3}^{n(h)},\nabla^2{V_3}^{n(h)},\nabla^3{V_3}^{n(h)}\right)$ depends only on $L^\infty(\omega)$-norms of $\nabla {V_3}^{n(h)}$, $\left(\nabla^2{V_3}^{n(h)}\right)^*$ and $\nabla^3{V_3}^{n(h)}$.  

As above, we can always find $h_2>0$ such that for all $h\leq h_2$:
$$O\left(\min(h^{\gamma-2},h^2)\big|F\big(\nabla {V_3}^{n(h)},(\nabla^2{V_3}^{n(h)})^*,\nabla^3{V_3}^{n(h)}\big)\big|\right)<\frac{\epsilon}{2},$$
which is possible by slowing down the rate of convergence of the sequence $V_3^{n(n)}$ (by reparametrizing the sequence).  Take $h_3=\min(h_1,h_2)$.  Then, for the first term on the right hand side of (\ref{trineq}), we have: 
\begin{equation*}
	\left|\frac{1}{h^{\gamma+2}}I_W^h(u_h)-\mathcal{I}_f(V_3^{n(h)})\right|<\frac{\epsilon}{2},
\end{equation*}
for all $h\leq h_3$.

Also, by (\ref{If}) and (\ref{Q2}), we obtain for the second term on the right hand side of (\ref{trineq}):
\begin{align*}
	\left|\mathcal{I}_f(V_3^{n(h)})-\mathcal{I}_f(V_3)\right|&\leq\frac{1}{24}\int_{\omega}\Big|\mathcal{Q}_2(\nabla^2 V_3^{n(h)}+\left({\rm sym}B(x')\right)_{2\times2})\\
	&\quad-\mathcal{Q}_2(\nabla^2 V_3+\left({\rm sym}B(x')\right)_{2\times2})\Big|\,dx'\\
&=\frac{1}{24}\int_{\omega}\Big|\mathcal{Q}_3\Big(\big(-\nabla^2V_3^{n(h)}-({\rm sym}B)_{2\times2}\big)^*\\
&\quad+{\rm sym}\Big(\big(c\big(-\nabla^2 V_3^{n(h)}-({\rm sym} B)_{2\times2}\big)\big)\otimes e_3\Big)\Big)\\
&\quad-\mathcal{Q}_3\Big(\big(-\nabla^2V_3-({\rm sym}B)_{2\times2}\big)^*\\
&\quad+{\rm sym}\left(\left(c\left(-\nabla^2 V_3-({\rm sym} B)_{2\times2}\right)\right)\otimes e_3\right)\Big)\Big|\,dx'.
\end{align*}
	Define: 
\begin{align*}
T^{n(h)}&=\big(-\nabla^2V_3^{n(h)}-({\rm sym}B)_{2\times2}\big)^*\\
	&\quad+{\rm sym}\Big(\big(c\big(-\nabla^2 V_3^{n(h)}-({\rm sym} B)_{2\times2}\big)\big)\otimes e_3\Big),
\end{align*} and: 
\begin{align*}
T&=\big(-\nabla^2V_3-({\rm sym}B)_{2\times2}\big)^*\\
&\quad+{\rm sym}\left(\left(c\left(-\nabla^2 V_3-({\rm sym} B)_{2\times2}\right)\right)\otimes e_3\right),
\end{align*} 
and notice that: $$\norm{T^{n(h)}}_{L^2(\omega)}+\norm{T}_{L^2(\omega)}\leq C.$$
Then, using H\"older's inequality and properties of $\mathcal{Q}_3$, we obtain:
\begin{align*}
	&\left|\mathcal{I}_f(V_3^{n(h)})-\mathcal{I}_f(V_3)\right|\\
	&\leq\frac{1}{24}\int_{\omega}\Big|\mathcal{Q}_3(T^{n(h)})-\mathcal{Q}_3(T)\Big|\,dx'\\
	&\leq C\big|\big|T^{n(h)}-T\big|\big|_{L^2(\omega)}\big(\big|\big|T^{n(h)}\big|\big|_{L^2(\omega)}+\big|\big|T\big|\big|_{L^2(\omega)}\big)\\
	&\leq C\big|\big|T^{n(h)}-T\big|\big|_{L^2(\omega)},
\end{align*}
Now, by (\ref{smoothapprox}), there exists $h_4>0$ such that
\begin{align*}
	\norm{T^{n(h)}-T}_{L^2(\omega)}&=\big|\big|(\nabla^2V_3^{n(h)}-\nabla^2V_3)^*\\
	&\quad+{\rm sym}\big(\big(c\big(\nabla^2 V_3^{n(h)}-\nabla^2 V_3\big)\big)\otimes e_3\big)\big|\big|_{L^2(\omega)}\\
	&\leq\big|\big|\big(\nabla^2V_3^{n(h)}-\nabla^2V_3\big)^*\big|\big|_{L^2(\omega)}\\
	&\quad+\big|\big|{\rm sym}\big(\big(c\big(\nabla^2 V_3^{n(h)}-\nabla^2 V_3\big)\big)\otimes e_3\big)\big|\big|_{L^2(\omega)}\\
	&<\frac{\epsilon}{2C},
\end{align*}
for all $h\leq h_4$.  Finally, taking $h^*=\min(h_3,h_4)$, we obtain that for all $h\leq h^*$ we have: $$\Big|\frac{1}{h^{\gamma+2}}I_W^h(u_h)-\mathcal{I}_f(V_3)\Big|<\epsilon. $$  This concludes the proof of (iii) in Theorem~\ref{Th-recovseq}.
\end{proof}
\section*{Acknowledgment.}
The authors would like to thank Marta Lewicka for bringing this problem to their attention and for useful discussions and feedback.

\bibliographystyle{plain}
\bibliography{PrestrainedElasticity}

\end{document}